\documentclass{amsart}
\usepackage{amsmath, amsthm, amsfonts, amssymb}
\usepackage{mathrsfs,graphicx}
\providecommand{\noopsort[1]{}}
\usepackage{bbm}
\usepackage{dsfont}
\usepackage{ifthen}
\numberwithin{equation}{section}
\usepackage{a4}
\usepackage[active]{srcltx}
\usepackage{verbatim}
\usepackage{hyperref}
\usepackage{enumerate}

\newtheorem{thm}{Theorem}[section]

\newtheorem{prop}[thm]{Proposition}
\newtheorem{lem}[thm]{Lemma}

\theoremstyle{remark}
\newtheorem{rem}[thm]{Remark}

\newtheorem{hyp}[thm]{Hypothesis}
\newtheorem{example}[thm]{Example}

\theoremstyle{definition}
\newtheorem{defn}[thm]{Definition}

\newcommand{\coloneqq}{\mathrel{\mathop:}=}
\newcommand{\applied}[2]{\langle #1,#2\rangle}

\renewcommand{\Re}{{\rm Re}\,}

\newcommand{\dist}{\mathrm{dist}}
\newcommand{\eps}{\varepsilon}
\newcommand{\lh}{\mathrm{span}}
\newcommand{\one}{\mathbbm{1}}

\newcommand{\weak}{\rightharpoonup}

\newcommand{\form}[3]{\ifthenelse{\equal{#2}{}}{\mbox{$ #1\Big[\, \cdot\,  , \, \cdot\,  \Big]$}}{
\mbox{$ #1\Big[ #2 , #3 \Big]$}}}
\newcommand{\qform}[2]{\ifthenelse{\equal{#2}{}}{\mbox{$ #1\Big[\cdot \Big]$}}{
\mbox{$ #1\Big[ #2 \Big]$}}}
\newcommand{\ip}[2]{\ifthenelse{\equal{#1}{}}{\mbox{$ \Big( \,\cdot\; \vline \; \cdot \, \Big) $}}{
\mbox{$ \Big( #1 \;  \vline \; #2 \Big)$}}}
\newcommand{\norm}[1]{\ifthenelse{\equal{#1}{}}{\mbox{$\|\cdot\|$}}{\mbox{$\| #1 \|$}}}
\newcommand{\betr}[1]{\ifthenelse{\equal{#1}{}}{\mbox{$|\cdot|$}}{\mbox{$| #1 |$}}}
\newcommand{\dual}[2]{\ifthenelse{\equal{#1}{}}{\mbox{$ \langle \,\cdot\; , \; \cdot \, \rangle $}}{
\mbox{$ \left\langle #1   ,  #2 \right\rangle$}}}
\newcommand{\pdual}[2]{\ifthenelse{\equal{#1}{}}{\mbox{$ \Big[ \,\cdot\; , \; \cdot \, \Big] $}}{
\mbox{$ \Big[ #1   ,  #2 \Big]$}}}
\newcommand{\bdual}[2]{\ifthenelse{\equal{#1}{}}{\mbox{$ \Big\langle \,\cdot\; , \; \cdot \, \Big\rangle_* $}}{
\mbox{$ \Big\langle #1 \;  , \; #2 \Big\rangle_*$}}}

\newcommand{\rg}{\mathrm{rg}}

\newcommand{\R}{\mathds{R}}

\newcommand{\N}{\mathds{N}}
\newcommand{\CR}{\mathds{R}}
\newcommand{\CC}{\mathds{C}}

\newcommand{\CN}{\mathds{N}}
\newcommand{\CZ}{\mathds{Z}}

\newcommand{\minus}{\,\mbox{-}\,}

\newcommand{\co}{\mathrm{co}}

\newcommand{\fix}{\mathrm{fix}}

\newcommand{\cL}{\mathscr{L}}
\newcommand{\bx}{\mathbf{x}}
\newcommand{\by}{\mathbf{y}}

\begin{document}

\title{Mean ergodic theorems on norming dual pairs}

\author{Moritz Gerlach}
\author{Markus Kunze}
\address{University of Ulm\\Institute of Applied Analysis\\89069 Ulm\\Germany}
\email{moritz.gerlach@uni-ulm.de, markus.kunze@uni-ulm.de}


\begin{abstract}
We extend the classical mean ergodic theorem to the setting of norming dual pairs.
It turns out that, in general, not all equivalences from the Banach space setting
remain valid in our situation.
However, for Markovian semigroups on the norming dual pair $(C_b(E),\mathscr{M}(E))$
all classical equivalences hold true under an additional assumption which is
slightly weaker than the e-property.
\end{abstract}

\maketitle 

\section{Introduction}

A power-bounded linear operator $T$ on a Banach space $X$ is called \emph{mean ergodic} 
if $\lim_{n\to \infty} A_nx$ exists for 
every $x \in X$. Here $A_n := n^{-1}\sum_{j=1}^{n-1}T^j$ are the Ces\`aro averages of the operator.

The classical mean ergodic theorem (see \cite[\S 2.1 Theorems 1.1 and 1.3]{krengel1985}) 
characterizes mean ergodic operators as follows.

\begin{thm}\label{t.classicerg}
Let $T$ be a power-bounded operator on a Banach space $X$. The following are equivalent.
\begin{enumerate}[(i)]
\item $T$ is mean ergodic.
\item $A_n x$ has a $\sigma(X,X^*)$-cluster point for all $x \in X$.
\item $\fix (T)$ separates $\fix (T^*)$, i.e.\ for all $0\neq x \in \fix (T^*)$ 
there exists a $x^* \in \fix (T^*)$ such that $\dual{x}{x^*}_*\neq 0$.
\item $X= \fix (T) \oplus \overline{\rg}^{\|\cdot\|}(I-T)$.
\end{enumerate}
\end{thm}

There are countless extensions of Theorem \ref{t.classicerg} to more general situations. 
These include weakening the assumption of power-boundedness, considering more general 
semigroups than the discrete semigroup $\{T^j\,:\, j \in \CN\}$, 
considering means other than the Ces\`aro averages and replacing the Banach space $X$ with a locally convex space $(X, \tau)$,
see e.g.\ \cite{eberlein1949, nagel1973, sato1978}.
An overview of these results and further references can be found in~\cite{krengel1985}.
Mean ergodic theorems for semigroups on locally convex spaces with additional assumptions are treated in~\cite{albanese2012, albanese2011}.


Even if the underlying space is a Banach space, it is not always reasonable to expect strong 
convergence of the means with respect to the norm topology.
An important example arises in the study of ergodic properties of Markov processes. 
Here, one works on the Banach space $\mathscr{M}(E)$ 
of bounded measures on the Borel $\sigma$-algebra of a Polish space $E$ or on the 
subset $\mathscr{P}(E)$ of probability measures.

Even though in some exceptional cases one obtains convergence of Ces\`aro averages 
(or even the semigroup itself) in the total variation norm \cite{seidler1997}, 
it is more natural to consider convergence in the weak topology induced by the bounded, continuous functions $C_b(E)$.

Unfortunately, it seems that one cannot treat this situation with a mean ergodic theorem on locally convex 
spaces $(X, \tau)$. The reason for this is that the known results require that the means be equicontinuous
with respect to $\tau$, see \cite{albanese2012,albanese2011,eberlein1949, sato1978}. If $\tau$ is
the weak topology $\sigma (\mathscr{M}(E), C_b(E))$, equicontinuity seems a rather strong assumption
which is not satisfied in interesting examples.

The literature on weak  Ces\`aro-convergence of Markov semigroups is rather extensive. 
Let us mention \cite{kps2010, sw2012,wh2011b, wh2011}.
However, a characterization of mean ergodicity in the spirit of Theorem \ref{t.classicerg} is still missing.
\medskip 

It is the purpose of the present article to fill this gap. We will work in the framework
of norming dual pairs introduced in \cite{kunze2009, kunze2011}
and consider simultaneously two semigroups which are related to each other via duality. 
From the point of view of applications to Markov semigroups this is rather natural, 
as associated with a Markov process there are two semigroups dual to each other. The first
acts on the space of bounded measurable functions on the state space $E$ (or a subspace thereof such as $C_b(E)$) 
and corresponds to the Kolmogorov backward equation and the second acts on the space of bounded measures on $E$ 
and corresponds to the Kolmogorov forward equation (or Fokker-Planck equation).

Throughout, we allow general (in particular also noncommutative) semigroups and means -- even though our main 
interest lies in Ces\`aro averages of one-parameter semigroups in discrete or continuous time -- and study convergence 
of the means in the weak topologies induced by the dual pair.

In our first main result (Theorem \ref{t.erg}), we show that in this general situation the statements corresponding 
to (i) and (ii) in Theorem \ref{t.classicerg} are equivalent and imply the statements corresponding to (iii) and (iv). 
We also provide counterexamples to show that in general (the statements corresponding to) (iii) does not imply (iv) 
and neither (iii) nor (iv) imply (i) and (ii).
\smallskip

Afterwards, we focus on the more special situation of Markovian semigroups on the norming dual pair
$(C_b(E), \mathscr{M}(E))$.
Besides others, our main assumption in this more special situation is a condition which is slightly weaker than 
the \emph{e-property} which played an important role in \cite{kps2010, sw2012}. 
Under that assumption we prove in Theorem \ref{t.eerg} that the statements corresponding 
to (i) -- (iv) in Theorem \ref{t.classicerg} for the semigroup on $\mathscr{M}(E)$ are all equivalent. 
Moreover, if the semigroup on $\mathscr{M}(E)$ is mean ergodic with respect to $\sigma (\mathscr{M}(E), C_b(E))$, 
then also the semigroup on $C_b(E)$ is mean ergodic even with respect 
to a topology finer than $\sigma (C_b(E), \mathscr{M}(E))$, namely the strict topology.
Considering semigroups on $(C_b(E),\mathscr{M}(E))$ rather than on the single Banach space $\mathscr{M}(E)$
makes our assumption natural, in fact, it is necessary for the convergence we obtain.
\medskip 

This article is organized as follows. In Section \ref{s.ndp} we recall some basic definitions and results 
about norming dual pairs. In Section \ref{s.as}, we introduce the notion of an ``average scheme'' which will 
act as our means. Afterwards, we take up our main line of study. 
First, we analyze convergence of average schemes on general norming dual pairs in Section \ref{s.convergence}, 
then the convergence of average schemes on $(C_b(E), \mathscr{M}(E))$ under additional assumptions in 
Section \ref{s.eproperty}. The concluding Section \ref{s.examples} contains our Counterexamples.

\section{Norming dual pairs}\label{s.ndp}

A \emph{norming dual pair} is a triple $(X,Y, \dual{}{})$ where $X$ and $Y$ are Banach spaces and $\dual{}{}$ is a 
duality between $X$ and $Y$ such that
\[
 \|x\| = \sup\{|\dual{x}{y}|: y \in Y,\|y\| \leq 1\}
\quad\mbox{and}\quad \|y\| = \sup\{|\dual{x}{y}|: x \in X,\|y\| \leq 1\} \, .
\]
Identifying $y$ with the linear functional $x \mapsto \dual{x}{y}$, we see that $Y$ is isometrically isomorphic 
with a norm closed subspace of $X^*$, the norm dual of $X$, which is norming for $X$. If the duality paring is understood, we will 
briefly say that $(X,Y)$ is a norming dual pair.\medskip 

Let us give some examples of norming dual pairs. If $X$ is a Banach space with norm dual $X^*$, then $(X,X^*)$ and thus, 
by symmetry, also $(X^*,X)$ is a norming dual pair with respect to the canonical duality $\dual{}{}_*$. 
If $(E, \Sigma)$ is a measurable space, we write $B_b(E)$ for the space of bounded, measurable functions 
on $(E, \Sigma)$ and $\mathscr{M}(E)$ for the space of complex measures on $(E, \Sigma)$. 
The space $B_b(E)$ is endowed with the supremum norm and the space $\mathscr{M}(E)$ is endowed with the 
total variation norm. Then $(B_b(E), \mathscr{M}(E))$ is a norming dual pair with respect to the duality
\[
 \dual{f}{\mu} := \int_E f\, d\mu\, .
\]
If $E$ is a Polish space, i.e.\ a topological space which is metrizable through a complete, separable 
metric, and $\Sigma$ is the Borel $\sigma$-algebra, then also $(C_b(E), \mathscr{M}(E))$ is a norming dual pair.
For the easy proofs of these facts we refer to \cite[Section 2]{kunze2011}.\medskip 

In what follows, we will be interested in locally convex topologies which are \emph{consistent (with the duality)}. 
We recall that a locally convex topology $\tau$ on $X$ is consistent if $(X,\tau)' = Y$, i.e.\ every $\tau$-continuous 
linear functional $\varphi$ on $X$ is of the form $\varphi (x) = \dual{x}{y}$ for some $y\in Y$. Of particular importance 
are the \emph{weak topologies} $\sigma (X,Y)$ and $\sigma (Y,X)$ associated with the dual pair. To simplify notation, 
in what follows we will write $\sigma$ for the $\sigma (X,Y)$ topology on $X$ and $\sigma'$ for the $\sigma (Y,X)$ topology
on $Y$. We will write $\weak$, resp.\ $\weak'$, to indicate convergence with respect to $\sigma$, resp.\ $\sigma'$.
\medskip

If $\tau$ is a topology on $X$, we write $\cL (X,\tau)$ for the algebra of $\tau$-continuous linear operators 
on $X$. We write $\cL (X)$ shorthand for $\cL (X,\|\cdot\|)$. By \cite[Prop 3.1]{kunze2011}, $\cL (X,\sigma)$ 
is a subalgebra of $\cL (X)$ which is closed in the operator norm. Moreover, identifying $Y$ with a closed 
subspace of $X^*$, an operator $S\in\cL (X)$ belongs to $\cL (X,\sigma)$ if and only if its norm adjoint $S^*$ leaves $Y$
invariant. In that case, the $\sigma$-adjoint of $S$, denoted by $S'$, is precisely $S^*|_Y$ and $\|S\|=\|S'\|$.\medskip 

Let us give a description of the operators in $\cL (X,\sigma)$ in the case where $X= B_b(E)$ or 
$X= C_b(E)$ is in canonical duality with $Y= \mathscr{M}(E)$. 

We recall that a \emph{bounded kernel} 
on a measurable space $(E,\Sigma)$ is a mapping $k: E\times \Sigma \to \CC$ such that (i) $k(x,\cdot)$ is a 
complex measure on $(E, \Sigma)$ for all $x \in E$, (ii) $k(\cdot, A)$ is $\Sigma$-measurable for all $A \in \Sigma$
and (iii) $\sup_{x\in E}|k|(x,E) < \infty$, where $|k|(x,\cdot)$ denotes the total variation of $k(x,\cdot)$.

A linear operator $S$ on a closed subspace $X$ of $B_b(E)$ is called a \emph{kernel operator (on $X$)} if there 
exists a bounded kernel $k$ on $(E,\Sigma)$ such that
\[
 (Sf)(x) = \int_Ef(y)k(x,dy), \quad \mbox{for all}\, f \in X.
\]
Note that if $X$ is $\sigma(B_b(E), \mathscr{M}(E))$-dense in $B_b(E)$, then $k$ is uniquely determined by $S$. In 
this case, $S$ has a unique extension to $B_b(E)$ and its $\sigma$-adjoint is given by
\[
 (S'\mu)(A) = \int_E k(x,A)\, d\mu (x)\quad \forall\, \mu \in \mathscr{M}(E)\, .
\]

It was seen in \cite[Prop 3.5]{kunze2011} that on the norming dual pair $(X,\mathscr{M}(E))$, where
$X=B_b(E)$ or, if $E$ is Polish and $\Sigma$ is the Borel $\sigma$-algebra, $X=C_b(E)$, an operator $S \in \cL (X)$
belongs to $\cL (X,\sigma)$ if and only if it is a kernel operator.

\section{Average Schemes}\label{s.as}

For a family $\mathscr{S}$ of linear operators on a vector space $X$ we denote by
\[ \fix(\mathscr{S}) \coloneqq \bigcap_{S\in\mathscr{S}} \ker(I-S) \]
its \textit{fixed space} and by
\[ \rg(I-\mathscr{S}) \coloneqq \{ x-Sx : x\in X,\, S\in\mathscr{S} \}\]
the range of $I-\mathscr{S}$.
	Moreover, for $x\in X$ we define
	\[ \co (\mathscr{S}x) \coloneqq \biggr\{ \sum_{k=1}^n a_k S_k x : a_k \geq 0,\,
	\sum_{k=1}^n a_k = 1,\, S_k \in \mathscr{S}, \, n\in \CN \biggr\},\]
the  convex hull of the orbit of $x$ under $\mathscr{S}$.
A family $\mathscr{S}$ containing the identity is called
a \emph{semigroup} if $ST\in \mathscr{S}$ for all $S, T\in \mathscr{S}$.

Inspired by \cite{eberlein1949} we make the following definition.

\begin{defn}
\label{d.as}
	Let $(X,Y)$ be a norming dual pair. An \emph{average scheme on $(X,Y)$} is a pair
	$(\mathscr{S},\mathscr{A})$, where $\mathscr{S}\subset\mathscr{L}(X,\sigma)$ is a semigroup
	with adjoint $\mathscr{S}' \coloneqq \{ S' : S\in \mathscr{S}\}$
	and $\mathscr{A}=(A_\alpha)_{\alpha\in\Lambda}\subset \mathscr{L}(X,\sigma)$ is a net of
	$\sigma$-continuous operators such that the following assertions
	are satisfied.
	\begin{enumerate}[(AS2)]
	\item[(AS1)] There exists $M>0$ such that $\Vert A_\alpha \Vert \leq M$ for all $\alpha\in\Lambda$.
	\item[(AS2)] $A_\alpha x \in \overline{\co}^\sigma(\mathscr{S}x)$ 
		and $A'_\alpha y \in \overline{\co}^{\sigma'}(\mathscr{S}'y)$
		for all $\alpha \in \Lambda$, $x\in X$ and $y\in Y$.
	\item[(AS3)] For every $S\in\mathscr{S}$ and all $x\in X$ and $y\in Y$ one has that
		\[ \lim_{\alpha} A_\alpha(S-I)x = \lim_{\alpha} (S-I)A_\alpha x = 0 \]
		and 
		 \[\lim_{\alpha} A'_\alpha(S'-I)y = \lim_{\alpha} (S'-I)A'_\alpha y = 0 \]
		 in the norm topology of $X$ and $Y$ respectively.
	\end{enumerate}
\end{defn}

We should point out that our terminology is somewhat different from that in \cite{eberlein1949}. In the language 
of Eberlein, the net $A_\alpha$ would be called a \emph{system of almost invariant integrals} and a semigroup 
$\mathscr{S}$ possessing such a system would be called \emph{ergodic}.
Moreover, we should note that there is no equicontinuity assumption for the 
averages $A_\alpha$ with respect to $\sigma$ or with respect to any other \emph{consistent} topology. Instead, 
we assume in (AS1) equicontinuity only with respect to the (in general not consistent) norm topology. On the 
other hand, in (AS3), we assume convergence in the norm topology, which is a stronger assumption than $\sigma$-convergence 
(and also than convergence with respect to a consistent topology on $X$.)

\begin{rem}
	\label{r.fixedpoint}
	We will frequently make use of the following observation.
	If $(\mathscr{S},\mathscr{A})$ is an average scheme on a norming dual pair $(X,Y)$ and 
	$x\in \fix(\mathscr{S})$, then $\overline{\co}^\sigma(\mathscr{S}x) = \{ x\}$ and hence, by (AS2),
	$A_\alpha x = x$ for all $\alpha \in \Lambda$.
\end{rem}

We now give some typical examples of average schemes. Throughout, $(X,Y)$ denotes a norming dual pair.

\begin{example}\label{ex.cesaro}
	Let $\mathscr{S} \coloneqq \{ S^k : k\in \CN_0\}$ be a semigroup that consists of powers of a single 
	operator $S\in \mathscr{L}(X,\sigma)$ and denote by
	\[ A_n \coloneqq \frac{1}{n} \sum_{k=0}^{n-1} S^k  \quad (n\in\CN)\]
	its \emph{Ces\`aro averages}. Assume that $\lim_{n\to\infty} \frac{1}{n}S^n x=0$ for all $x\in X$ and
that $\lim_{n\to\infty}\frac{1}{n}(S')^ny =0$ for all $y \in Y$. Moreover, assume that there
	exists $M>0$ such that $\Vert A_n \Vert < M$ for all $n\in\N$, i.e.\ $S$ is \emph{Ces\`aro bounded}. 
	Both assumptions are satisfied if $S$ is \emph{power-bounded}, i.e.\ $\sup_{n\in \CN} \|S^n\| < \infty$.
	
	Clearly, (AS1) and (AS2) are satisfied. As for (AS3), we have
	\[ \lim_{n\to\infty} A_n (S-I)x = \lim_{n\to\infty} \frac{1}{n} (S^n-I)x = 0 \text{ for all }x\in X\]
	and, similarly, $\lim_{n\to\infty} A'_n (S'-I)y  = 0$ for all $y\in Y$.
	Thus,
	 $(\mathscr{S},(A_n)_{n\in\CN})$ is an average scheme.
\end{example}

\begin{example}\label{ex.abel}
We again consider $\mathscr{S} \coloneqq \{ S^k : k\in \CN_0\}$ for an operator $S \in \cL (X, \sigma)$.
If $S$ has spectral radius $r(S) = \lim_{n\to\infty} \|S^n\|^{\frac{1}{n}} \leq 1$,
then for $r \in [0,1)$ the series $\sum_{k=0}^\infty r^kS^k$ converges in 
operator norm and thus represents an element of $\cL (X,\sigma)$. We denote by 
\[ A_rx := (1-r)\sum_{k=0}^\infty r^kS^kx \quad (r \in [0,1))\]
the \emph{Abel averages} of $S$. If $M :=\sup_{0\leq r< 1} \|A_r\| < \infty$, then $S$ is called \emph{Abel bounded}.
Note that power-bounded operators are Abel bounded.

For an Abel bounded operator $S \in \cL (X,\sigma)$,
the pair $(\mathscr{S}, (A_r)_{r\in [0,1)})$ is an average scheme.

Indeed, (AS1) is clear. As for (AS2) we see that $A_r = \lim_{n\to\infty} \frac{1-r}{1-r^{n+1}}\sum_{k=0}^n
r^kS^k$ in operator norm. Hence $A_rx$ belongs even to the norm closure of $\mathrm{co}(\mathscr{S} x)$. For the 
$\sigma$-adjoint, one argues similarly. It remains to verify (AS3). So that end, note that
\[
 \|A_rSx - A_rx\| = (1-r)\|x-A_rSx\| \leq (1-r)\big[1 + M\|S\|\big]\|x\| \to 0
\]
as $r \uparrow 1$. On $Y$, one argues similarly.
\end{example}

\begin{example}
\label{ex.integrable}
Let $\mathscr{S} \coloneqq \{ S(t)\,:\, t \geq 0\} \subset \cL (X,\sigma)$ 
be an integrable semigroup on $(X,Y)$, cf.\ \cite{kunze2011}. This means that $S(0)$ is the identity on $X$ and 
for $t,s \geq 0$, we have $S(t+s) = S(t)S(s)$. Moreover, there exists $M \geq 1$ and $\omega \in \CR$ such that
$\|S(t)\| \leq Me^{\omega t}$ for all $t \geq 0$. Finally, for all $x\in X$ and $y\in Y$ the function $t \mapsto 
\dual{S(t)x}{y}$ is measurable and for some (equivalently, all) $\lambda$ with $\Re\lambda > \omega$ there 
exists an operator $R(\lambda) \in \cL (X,\sigma)$ such that
\[
 \dual{R(\lambda)x}{y} = \int_0^\infty e^{-\lambda t}\dual{S(t)x}{y}\, dy, \quad \mbox{for all}\,\, x \in X, y \in Y.
\] 
It follows from \cite[Thm 5.8]{kunze2011} that if $\mathscr{S}$ is an integrable semigroup, then for every $t >0$
there exists an operator $A_t \in \cL (X,\sigma)$ such that
\[
 \dual{A_tx}{y} = \frac{1}{t}\int_0^t\dual{S(s)x}{y}\, ds\, .
\]
We call the semigroup $\mathscr{S}$ Ces\`aro bounded if $M:=\sup_{t>0}\|A_t\| < \infty$. If $\mathscr{S}$ is an integrable,
Ces\`aro bounded semigroup such that $\frac{1}{t}S(t)x \to 0$ and $\frac{1}{t}S(t)'y \to 0$
as $t\to \infty$ for arbitrary $x \in X$ and $y \in Y$, then $(\mathscr{S}, (A_t)_{t>0})$ is an average scheme.

(AS1) is clear and (AS2) is a consequence of the Hahn-Banach theorem on the locally convex spaces $(X,\sigma)$ resp.\ 
$(Y,\sigma')$, cf.\ the end of the proof of Theorem 4.4 in \cite{kunze2011}. As for (AS3), we note that for $t>0$ and $s\geq 0$
we have
\[
 A_tS(s) - A_t = \frac{s}{t}(I - S(t))A_s = \frac{s}{t}A_s(I-S(t))
\]
as is easy to see using the semigroup law. Consequently, for every $x \in X$ we have
$\|A_tS(s)x - A_tx\| \leq s M(\|x\|t^{-1} + \|t^{-1}S(t)x\|) \to 0$ as $t \to \infty$. On $Y$, one argues similarly.

In particular, $(\mathscr{S},(A_t)_{t >0})$ is an average scheme whenever the integrable semigroup $\mathscr{S}$ is bounded.
\end{example}

Concerning the last example, let us note that if $\mathscr{S}$ is an integrable semigroup, then the operators 
$R(\lambda)$ form a pseudo resolvent, hence, there is a unique, possibly multivalued operator $\mathscr{G}$ with 
$R(\lambda ) = (\lambda -\mathscr{G})^{-1}$, the \emph{generator of $\mathscr{S}$}. 
In this case, as a consequence of \cite[Prop 5.7]{kunze2011}, $\fix (\mathscr{S}) = \mathrm{ker}\mathscr{G} = 
\{ x \in X\,:\, (x,0) \in \mathscr{G}\}$. 

For more information 
about integrable semigroups and their generators, we refer to \cite{kunze2011}.

\section{Convergence of average schemes}\label{s.convergence}

We start with the definition of weak ergodicity.

\begin{defn}
	We say that an average scheme $(\mathscr{S},\mathscr{A})$ on a norming
	dual pair $(X,Y)$ is \emph{weakly ergodic},
 if the $\sigma$-limit of $(A_\alpha x)$ exists for every $x\in X$ and the
$\sigma'$-limit of $(A_\alpha'y)$ exists for every $y\in Y$.
\end{defn}

In the mean ergodic theorem on norming dual pairs we need a slightly stronger version of assertion (ii)
of Theorem~\ref{t.classicerg}. This is due to the fact that the strategy for the proof differs from
the classical one since not all assertions corresponding to (i) -- (iv) are equivalent in our situation.
We use the following terminology.

\begin{defn}
	\label{d.compactnet}
	We say that a net $(x_\alpha)_{\alpha\in\Lambda}$ in a topological space $E$ \emph{clusters} if
	every subnet of $(x_\alpha)$ has a cluster point, i.e.\ it has a convergent subnet.
\end{defn}

A net clusters whenever the set of its elements is relatively compact.
However, if a net $(x_\alpha)_{\alpha\in\Lambda}$ clusters, 
one cannot infer that the set $\{ x_\alpha : \alpha\geq\alpha_0 \}$ is
relatively compact for some $\alpha_0 \in \Lambda$.
For a sequence, these two properties are equivalent, which is probably well-known.
However, we were not able to find a reference and hence present the short proof
for the sake of completeness.

\begin{lem}
\label{l.seqcompact}
	Let $(x_n)$ be a sequence in a  topological vector space $(X,\tau)$ 
	such that every subsequence of $(x_n)$ has a convergent subnet.
	Then the set $\mathscr{M}\coloneqq \{ x_n : n\in\N\}$ is relatively compact.
\end{lem}
\begin{proof}
	In view of \cite[\S 5.6(2)]{koethe1969}, it suffices to show that $\mathscr{M}$ is totally bounded.
	Assume the converse. Then there exists an open neighborhood of the origin $U$ and a subsequence 
	$y_k \coloneqq x_{n_k}$ such that
	\[ y_k \not\in \bigcup_{j=1}^{k-1} (U+y_j)\]
	for all $k\in\N$. By assumption, $(y_k)$ contains a convergent subnet 
	$(y_{k(\alpha)})_{\alpha\in A}$
	whose limit we denote by $y$.
	Now, we choose a circled neighborhood of the origin $V$ such that $V+V\subset U$,
	which exists by \cite[\S 15.1(3)]{koethe1969}.
	Then there is a $\beta\in A$ such that $y_{k(\alpha)} \in V+y$ for all $\alpha\geq \beta$
	and hence, $y \in V+y_{k(\beta)}$. This implies that
	\[ y_{k(\alpha)} = y_{k(\alpha)}-y+y \in V+y \subset V+V+y_{k(\beta)} \subset U+y_{k(\beta)} 
	\subset \bigcup_{j=1}^{k(\beta)}( U+y_j).\]
	for all $\alpha\geq \beta$. Since $\{k(\alpha):\alpha\in A\}$ is cofinal in $\N$, 
	this is in contradiction to
	the construction of the sequence $(y_k)$. Hence, $\mathscr{M}$ is relatively compact.
\end{proof}

The following is the main result of this section.

\begin{thm}\label{t.erg}
Let $(\mathscr{S},\mathscr{A})$ be an average scheme on a norming dual pair $(X,Y)$.
Then the following are equivalent:
\begin{enumerate}[(i)]
 \item\label{t.erg.1} The average scheme $(\mathscr{S},\mathscr{A})$ is weakly ergodic.
 \item\label{t.erg.2} For every $x \in X$ the net $(A_\alpha x)$ clusters in $(X,\sigma)$ and
for every $y \in Y$ the net $(A'_\alpha y)$ clusters in $(Y,\sigma')$.
\end{enumerate}
If these equivalent conditions are satisfied, then
\begin{enumerate}[(i)]
\setcounter{enumi}{2}
 \item\label{t.erg.3} The fixed spaces $\fix (\mathscr{S})$ and $\fix (\mathscr{S}')$ separate each other.
 \item\label{t.erg.4} We have $X= \fix (\mathscr{S}) \oplus \overline{\lh}^\sigma \rg(I-\mathscr{S})$ and 
$Y= \fix (\mathscr{S}') \oplus \overline{\lh}^{\sigma'}\rg(I-\mathscr{S}')$.
 \item\label{t.erg.5} The operator $P$, defined by $Px \coloneqq \sigma\minus\lim_{\alpha} A_\alpha x$ belongs to 
$\cL (X,\sigma)$ and the $\sigma$-adjoint $P'$ of $P$ is given by
$P'y  =\sigma'\minus\lim_{\alpha} A'_\alpha y$ for all $y\in Y$.
Moreover, $P$ is the projection onto $\fix (\mathscr{S})$ along $\overline{\lh}^\sigma \rg(I-\mathscr{S})$, $P'$
the projection onto $\fix(\mathscr{S}')$ along $\overline{\lh}^{\sigma'}(I-\mathscr{S}')$
and $PS=SP=P$ for all $S\in\mathscr{S}$.
\end{enumerate}
\end{thm}

For a weakly ergodic average scheme $(\mathscr{S},\mathscr{A})$,
the operator $P$ from (\ref{t.erg.5}) is called the \emph{ergodic projection}.
Note that the ergodic projection $P$ 
is uniquely determined by the semigroup $\mathscr{S}$ and 
independent of the averages $\mathscr{A}$.

We prepare the proof of Theorem \ref{t.erg} through a series of lemmas.

\begin{lem}\label{l.1}
Let $X$ be a Banach space and $\mathscr{S}\subset \mathscr{L}(X)$ be semigroup of bounded operators on $X$.
Moreover, let $(A_\alpha)_{\alpha \in\Lambda}\subset\mathscr{L}(X)$ be a net and let $x\in X$ be
such that
\[ \lim_{\alpha} (S-I)A_\alpha x  = 0 \text{ for all }S \in \mathscr{S}.\]
Assume that $Z \subset X^*$ separates points in $X$ and $S^*Z \subset Z$ for all $S\in\mathscr{S}$.
Then every $\sigma (X,Z)$-cluster point of $(A_\alpha x)$ belongs to $\fix (\mathscr{S})$.
\end{lem}

\begin{proof}
Fix $x \in X$ and let $w$ be a $\sigma (X,Z)$-cluster point of $(A_\alpha x)$.
Let $S\in\mathscr{S}$.
We have
\[
 Sw-w = (S-I)(w-A_\alpha x) + (S-I) A_\alpha x
\]
 for all $\alpha \in \Lambda$
and, by assumption, $(S-I) A_\alpha x \to 0$ in norm and hence also with respect to $\sigma(X, Z)$.

Now fix $z \in Z$. Given $\eps>0$, we find $\alpha_0$ such that 
$\vert \applied{(S-I)A_\alpha x}{z} \vert < \eps$ for all $\alpha \geq \alpha_0$.
Since $S^* z \in Z$ and since 
$w$ is an $\sigma (X,Z)$-cluster point of $(A_\alpha x)$, we find some $\beta \geq \alpha_0$ such that
\[
 |\applied{(S-I)(w-A_\beta x)}{z}| = |\applied{w-A_\beta x}{S^*z-z}| \leq \eps\, .
\]
This implies that $|\applied{Sw-w}{z}| \leq 2\eps$. Since $\eps>0$ was arbitrary, it follows that
$\applied{Sw-w}{z} =0$
and thus, since $z \in Z$ was arbitrary, $w = Sw$.
\end{proof}

\begin{lem}\label{l.norm}
Let $(\mathscr{S},\mathscr{A})$ be an average scheme on a norming dual pair $(X,Y)$.
Then
\[
 X_0 \coloneqq \{ x \in X \,:\, \lim_{\alpha} A_\alpha x\text{ exists w.r.t.\ }\Vert\cdot\Vert\,\}
\]
is a norm-closed subspace of $X$ and invariant under the action of $\mathscr{S}$.
Moreover, the sum $\fix(\mathscr{S})+\overline{\lh}^{\Vert\cdot\Vert}\rg(I-\mathscr{S})$
is direct and
$\fix (\mathscr{S}) \oplus \overline{\lh}^{\Vert\cdot\Vert}\rg(I-\mathscr{S})\subset X_0$.
Finally, $P_0 x\coloneqq \Vert\cdot\Vert\minus\lim_{\alpha} A_\alpha x$
defines a bounded operator on $X_0$ which is a projection onto $\fix(\mathscr{S})$ 
with $\overline{\lh}^{\Vert\cdot\Vert}\rg(I-\mathscr{S})\subset \ker P_0$.
\end{lem}

\begin{proof}
Let $x_k \in X_0$ and $\lim x_k = x$ with respect to $\|\cdot\|$. 
We denote by $P_0x_k$ the limit of $A_\alpha  x_k$ for every fixed $k\in \CN$.
Then $P_0x_k \in \fix (\mathscr{S})$ by Lemma~\ref{l.1} and we have
\[
 \|P_0x_k - P_0x_l\| = \lim_{\alpha} \|A_\alpha (x_k-x_l)\| \leq M \|x_k-x_l\|
\]
 for all $k,l \in \N$,
where $M$ is such that $\|A_\alpha \|\leq M$ for all $\alpha\in\Lambda$. 
Since $(x_k)$ is a Cauchy sequence, so is $(P_0x_k)$.
Thus, $P_0x_k \to \bar{x}$ for some $\bar{x}$ which belongs to $\fix(\mathscr{S})$ as the latter is closed.
A $3\eps$-argument shows that $A_\alpha x \to \bar{x}$. It follows that $X_0$ is closed and $P_0x=\bar{x}$. 

We have seen that $\|P_0\|\leq M$ and $P_0 X_0 \subset \fix (\mathscr{S})$.
Conversely, $\fix (\mathscr{S})\subset P_0 X_0$ since $A_\alpha x \equiv x = P_0 x$ 
for $x \in \fix (\mathscr{S})$.
Hence, $P_0 X_0 = \fix(\mathscr{S})$ and $P_0$ is a projection.
The $\mathscr{S}$-invariance of $X_0$ follows from (AS3).

By the definition of an average scheme, $\lim_{\alpha} A_\alpha x = 0$ for all $x\in \rg(I-\mathscr{S})$.
In view of the uniform boundedness of the operators $A_\alpha$, this remains true for 
$x \in \overline{\lh}^{\Vert\cdot\Vert}\rg(I-\mathscr{S})$.
Since $A_\alpha x \to x$ for $x\in\fix (S)$, it follows that the sum of $\fix (S)$ 
and $\overline{\lh}^{\|\cdot\|}\rg(I-\mathscr{S})$ is direct and that 
$\fix (S) \oplus \overline{\lh}^{\|\cdot\|}\rg(I-\mathscr{S}) \subset X_0$.
\end{proof}

\begin{lem}\label{l.direct}
Let $(\mathscr{S},\mathscr{A})$ be an average scheme on a norming dual pair $(X,Y)$.
If $\fix (\mathscr{S})$ separates $\fix (\mathscr{S}')$, then 
$\fix (\mathscr{S}) + \rg(I-\mathscr{S})$ is $\sigma (X,Y)$-dense in $X$. 
If $\fix (\mathscr{S}')$ separates $\fix (\mathscr{S})$, then the sum 
$\fix (\mathscr{S}) + \overline{\lh}^\sigma \rg(I-\mathscr{S})$ is direct. 
\end{lem}

\begin{proof}
Assume that $\fix (\mathscr{S})$ separates $\fix (\mathscr{S}')$. Let $y \in Y$ be such that 
$\applied{x}{y} = 0$ for all $x \in \fix (\mathscr{S}) \oplus \rg(I-\mathscr{S})$.
Then, in particular, $0= \applied{x-Sx}{y} = \applied{x}{y-S'y}$ for all $x \in X$ and $S\in\mathscr{S}$.
Since $X$ separates $Y$, it follows that $y = S'y$ for all $S\in\mathscr{S}$, i.e.\ $y \in \fix (\mathscr{S}')$. 
Moreover, $\applied{x}{y} =0$ for all $x \in \fix (\mathscr{S})$. 
By assumption, this implies $y=0$. It now follows from the 
Hahn-Banach theorem, applied on the locally convex space $(X,\sigma)$,
that $\fix (\mathscr{S}) + \rg(I-\mathscr{S})$ is $\sigma (X,Y)$-dense in $X$.

Now assume that $\fix (\mathscr{S}')$ separates $\fix (\mathscr{S})$. Since every $y\in\fix(\mathscr{S}')$
vanishes on $\rg(I-\mathscr{S})$, it also vanishes on $\overline{\lh}^\sigma \rg(I-\mathscr{S})$ by linearity and
continuity. Thus, $\applied{x}{y}=0$ for all $x \in \fix (\mathscr{S}) \cap \overline{\lh}^\sigma \rg(I-\mathscr{S})$
and $y\in\fix(\mathscr{S}')$.
As $\fix(\mathscr{S}')$ separates $\fix(\mathscr{S})$, it follows that $0$ is the only element of
$\fix (\mathscr{S}) \cap \overline{\lh}^\sigma \rg(I-\mathscr{S})$.
\end{proof}

\begin{lem}
\label{l.convergence}
Let $(\mathscr{S},\mathscr{A})$ be an average scheme on a norming dual pair $(X,Y)$ and
$\tau$ be a locally convex topology on $X$ finer than $\sigma$. Let
\[ X_1 \coloneqq \{ x\in X : (A_\alpha x)_{\alpha\in\Lambda} \text{ clusters in } (X,\tau) \}.\]
If $\fix (\mathscr{S}')$ separates $\fix (\mathscr{S})$, then
$\tau\minus\lim_\alpha A_\alpha x \in X$ exists  for all $x\in X_1$.
\end{lem}
\begin{proof}
Applying Lemma~\ref{l.norm} to the average scheme $(\mathscr{S}',\mathscr{A}')$ on $(Y,X)$, we find that
\[Y_0 \coloneqq \{y \in Y\,:\, \|\cdot\|\minus \lim_{\alpha} A'_\alpha y\,\,\mbox{exists}\}\]
is a norm-closed subspace of $Y$ that contains $\fix(\mathscr{S}')\oplus \rg(I-\mathscr{S}')$.
Moreover, there exists an operator $R_0 \in \cL (Y_0)$ such that 
$\|\cdot\|\minus \lim_{\alpha}A_\alpha'y = R_0y$ for all $y\in Y_0$.
As $\fix(\mathscr{S}')$ separates $\fix(\mathscr{S})$,  Lemma~\ref{l.direct} yields that $Y_0$ is $\sigma (Y,X)$-dense in $Y$. 
Hence, we may identify $X$ with a subspace of $Y_0^*$. Doing so, it follows that
$\sigma (Y_0^*, Y_0)\minus \lim_{\alpha} A_\alpha x = R_0^*x$ for all $x \in X$.
Let us fix $x \in X_1$ and choose an arbitrary subnet of $(u_\beta)$ of $(A_\alpha x)$.
By assumption, $(u_\beta)$ has a $\tau$-cluster point $\bar{x}\in X$.
Since $\bar{x}$ is also a $\sigma (Y_0^*,Y_0)$-cluster point of $(A_\alpha x)$, we infer that $\bar{x} = R_0^*x$.
Thus, every subnet of $(A_\alpha x)$ has a subnet converging to $R_0^*x \in X$ in $(X,\tau)$.
This implies that $\tau\minus\lim_{\alpha} A_{\alpha}x = R_0^*x$ for all $x\in X_1$.
\end{proof}

Now, we have the tools at hand to prove Theorem~\ref{t.erg}.

\begin{proof}[Proof of Theorem~\ref{t.erg}]
The implication (\ref{t.erg.1}) $\Rightarrow$ (\ref{t.erg.2}) is trivial, 
so assume that (\ref{t.erg.2}) holds.
Let us verify (\ref{t.erg.3}) first.
Since $Y$ is norming for $X$, given $x \in \fix (\mathscr{S})$, $x\neq 0$,
we find $y \in Y$ such that $\applied{x}{y} = a \neq 0$. 
By assumption, $A'_\alpha y$ has a $\sigma (Y,X)$-cluster point $z$ which, by Lemma~\ref{l.1}, is an 
element of $\fix (\mathscr{S}')$.
Since
\[ \applied{x}{A'_\alpha y} = \applied{A_\alpha x}{y} = \applied{x}{y}=a\]
for all $\alpha\in\Lambda$ it follows that $\applied{x}{z} = a \neq 0$.
Hence, $\fix (\mathscr{S}')$ separates $\fix (\mathscr{S})$. Interchanging 
the roles of $X$ and $Y$, it follows that $\fix (\mathscr{S})$ separates $\fix (\mathscr{S}')$.
Now, Assertion (\ref{t.erg.1}) follows immediately from Lemma~\ref{l.convergence} applied to the
average schemes $(\mathscr{S},\mathscr{A})$ and $(\mathscr{S}',\mathscr{A}')$ and the weak topologies
$\sigma(X,Y)$ and $\sigma(Y,X)$, respectively.\smallskip

We continue with the verification of Assertion (\ref{t.erg.4}). By (\ref{t.erg.3}) and Lemma~\ref{l.direct}, 
the sums 
\[ \fix (\mathscr{S}) + \overline{\lh}^\sigma\rg(I-\mathscr{S}) \text{ and }
\fix (\mathscr{S}') + \overline{\lh}^{\sigma'}\rg(I-\mathscr{S}')\] are direct and dense in $X$ with respect to 
$\sigma$, resp.\ dense in $Y$ with respect to $\sigma'$. 
Let $x\in X$ and $\bar{x} \coloneqq \lim_\alpha A_\alpha x \in \fix(\mathscr{S})$. Since 
$x-\mathrm{co}(\mathscr{S}x) \subset \lh \,\rg(I-\mathscr{S})$, we have that
$x-A_\alpha x \in \overline{\lh}^\sigma \rg(I-\mathscr{S})$ for all $\alpha\in\Lambda$ and, consequently,
\begin{equation}\label{eq.rangeconv}
 x - \bar{x} = \sigma\minus\lim_{\alpha} (x- A_\alpha x) \in \overline{\lh}^\sigma \rg(I-\mathscr{S}).
 \end{equation}
This shows that $X=\fix(\mathscr{S})\oplus \overline{\lh}^\sigma \rg(I-\mathscr{S})$ and analogously
we deduce that $Y= \fix(\mathscr{S}')\oplus \overline{\lh}^{\sigma'} \rg(I-\mathscr{S}')$. 
\medskip 

In order to verify Assertion (\ref{t.erg.5}), we consider the operators $P$ and $R$,
defined by $Px \coloneqq \sigma\minus\lim_\alpha A_\alpha x$
and $Ry \coloneqq \sigma\minus\lim_\alpha A'_\alpha y$. 
Since $\sup\{ \Vert A_\alpha \Vert : \alpha\in\Lambda\}<\infty$ and $X$ and $Y$ are norming for each other, 
it follows that $P\in\mathscr{L}(X)$ and $R\in \mathscr{L}(Y)$.
Moreover, for $x\in X$ and $y\in Y$ we have
\[
 \applied{Px}{y} = \lim_{\alpha}\applied{A_\alpha x}{y}
 = \lim_{\alpha}\applied{x}{A_\alpha'y} = \applied{x}{Ry}
\]
This implies that $P^*Y = RY \subset Y$, hence $P \in \cL (X, \sigma)$. 
As fixed points are  invariant under $\mathscr{A}$, it follows from Lemma~\ref{l.1} that 
$P$ is a projection with $\rg P = \fix(\mathscr{S})$ and, by (AS3),
$\lh\,\rg(I-\mathscr{S})\subset \ker P$. Since $P$ is $\sigma$-continuous
and $\ker P$ is closed, $\overline{\lh}^\sigma \rg(I-\mathscr{S})\subset \ker P$. 
The converse inclusion follows from \eqref{eq.rangeconv}.
Interchanging the roles of $X$ and $Y$, we see that $P'=R$ is the projection onto $\fix(\mathscr{S}')$
along $\overline{\lh}^{\sigma'}(I-\mathscr{S}')$. 

In view of $SPx = Px$ and $P(S-I)x = \lim_\alpha A_\alpha (S-I)x = 0$ for all $x\in X$ and $S\in\mathscr{S}$, 
$P$ commutes with every operator in $\mathscr{S}$.
\end{proof}


Theorem \ref{t.erg} is symmetric in $X$ and $Y$,
in that for every statement concerning $X$ resp.\ $\mathscr{A}$, there 
is a corresponding statement about $Y$ resp.\ $\mathscr{A}'$. This symmetry is crucial.
Indeed, in the case where $Y=X^*$, the norm-bounded net $(A_\alpha x)$ is always relatively $\sigma$-compact and hence
$\sigma$-clusters. However, it does not necessarily $\sigma$-converge as the example of the left shift
on $\ell^\infty$ shows.

Moreover, even if  $A_\alpha x$ $\sigma$-converges for all $x \in X$, one cannot deduce
$\sigma'$-con\-ver\-gence of $A_\alpha'y$ for all $y \in Y$ and hence no weak ergodicity of the average scheme, 
see Example~\ref{ex:Pnotcontinuous}.
\smallskip

Comparing Theorem~\ref{t.erg} with the classical mean ergodic theorem on Banach spaces, an immediate question 
is whether assertions (\ref{t.erg.1}) and (\ref{t.erg.2}) are also equivalent
with (\ref{t.erg.3}) and (\ref{t.erg.4}).  This is not the case in general, as Examples~\ref{ex:noconvergence} and \ref{ex.nodecomp} show.
However, some weaker results hold true, which are stated in the following proposition.

\begin{prop}\label{p.weakerg}
Let $(\mathscr{S},\mathscr{A})$ be an average scheme on a norming dual pair $(X,Y)$. 
\begin{enumerate}[(a)]
\item\label{i.weakerg.a} Suppose that
$\fix (\mathscr{S})$ and $\fix (\mathscr{S}')$ separate each other.
 If $\fix (\mathscr{S})$ (hence also $\fix (\mathscr{S}')$) is finite dimensional, then
\begin{equation}\label{eq.direct}
 X = \fix (\mathscr{S}) \oplus \overline{\lh}^\sigma \rg(I-\mathscr{S}) \quad\mbox{and}\quad
Y = \fix (\mathscr{S}') \oplus \overline{\lh}^{\sigma'}\rg(I-\mathscr{S}').
\end{equation}
 \item\label{i.weakerg.b} Now assume that \eqref{eq.direct} holds and let $P$ denote the projection onto 
$\fix (\mathscr{S})$ along $\overline{\lh}^\sigma \rg(I-\mathscr{S})$.
Then $P \in \cL (X, \sigma)$, $P'$ is the projection onto $\fix(\mathscr{S}')$ along 
$\overline{\lh}^{\sigma'}\rg(I-\mathscr{S}')$ and
\[
 \sigma (X, Y_0)\minus\lim_{\alpha}A_\alpha x = Px\qquad\mbox{and}\qquad
\sigma(Y,X_0)\minus\lim_{\alpha}A_\alpha 'y = P'y
\]
for all $x\in X$ and $y\in Y$, where 
\[ X_0 \coloneqq \fix (\mathscr{S}) \oplus \overline{\lh}^{\Vert\cdot\Vert}\rg(I-\mathscr{S})
\quad\mbox{and}\quad
Y_0 \coloneqq \fix (\mathscr{S}') \oplus \overline{\lh}^{\Vert\cdot\Vert}\rg(I-\mathscr{S}').\]
Moreover, $\fix(\mathscr{S})$ and $\fix(\mathscr{S}')$ separate each other.
\end{enumerate}
\end{prop}

\begin{proof}
(a) By Lemma~\ref{l.direct}, the sum  $\fix (\mathscr{S}) \oplus \overline{\lh}^\sigma\rg(I-\mathscr{S})$
is direct and $\sigma$-dense in $X$.
Since $\fix (\mathscr{S})$ is finite dimensional, the sum is $\sigma$-closed by \cite[\S15.5(3)]{koethe1969}.
It follows that the sum equals $X$.
Similarly one sees that $Y = \fix (\mathscr{S}') \oplus \overline{\lh}^\sigma \rg(I-\mathscr{S}')$.\medskip 

(b) For $x\in X$ we have $x= Px + (I-P)x \in \fix (\mathscr{S}) \oplus \overline{\lh}^\sigma\rg(I-\mathscr{S})$.
As $A_\alpha x \equiv x$ for fixed points $x$, to prove $A_\alpha x \to Px$ with respect to 
$\sigma (X, Y_0)$, it suffices to show that $\lim_\alpha \applied{A_\alpha x}{y} =0$ for 
$x \in \overline{\lh}^\sigma \rg(I-\mathscr{S})$ and $y \in Y_0$.
To that end, first observe that for $w=(I-S)v\in \rg(I-\mathscr{S})$ and $y \in \fix (\mathscr{S}')$ we have that
\[ \applied{w}{y} = \applied{(I-S)v}{y} = \applied{v}{(I-S')y} = 0,\]
i.e.\ $y$ vanishes on $\rg(I-\mathscr{S})$ and hence on 
$\overline{\lh}^\sigma \rg(I-\mathscr{S})$ by continuity.
Now let
$y =z-S'z\in \rg(I-\mathscr{S}')$.  Then 
\[
 \lim_\alpha \applied{A_\alpha x}{y} = \lim_\alpha \applied{(I-S)A_\alpha x}{z}  = 0
\]
by (AS3).
In view of the uniform boundedness of $(A_\alpha)$, property (AS1), it follows that 
$\lim_\alpha \applied{A_\alpha x}{y} = 0$ for $y \in \overline{\lh}^{\Vert\cdot\Vert}\rg(I-\mathscr{S}')$, too.
Altogether, we have proved that
\[ \sigma (X,Y_0)\minus\lim_{\alpha} A_\alpha x=Px.\]
It is easy to see that $P$ is norm-closed, hence it is 
bounded by the closed graph theorem.

By interchanging the roles of $X$ and $Y$, one sees that 
$\sigma(Y, X_0)\minus\lim_{\alpha} A_\alpha'y = Ry$ where $R\in \cL (Y)$ is the projection onto
$\fix(\mathscr{S}')$ along $\overline{\lh}^{\sigma'}\rg(I-\mathscr{S}')$.

Now let $x \in X$ and $y \in Y$ be given. Since $Px=P^2x\in \fix(\mathscr{S})\subset X_0$ 
and $Ry=R^2y\in \fix(\mathscr{S}') \subset Y_0$, it follows that
\begin{align*}
 \applied{Px}{y} & = \applied{P^2x}{y} = \lim_{\alpha} \applied{A_\alpha Px}{y}
= \lim_{\alpha} \applied{Px}{A_\alpha'y}\\
&= \applied{Px}{Ry}= \lim_{\alpha}\applied{A_\alpha x}{Ry}
=\lim_\alpha \applied{x}{A_\alpha' Ry}= \applied{x}{Ry}.
\end{align*}

This shows that $P^*Y\subset Y$ and $P'=P^*|_Y =R$. In particular, $P \in \cL (X, \sigma)$.

In order to prove the final assertion, let $x\in \fix(\mathscr{S})$, $x\neq 0$, and $y\in Y$ 
such that $\applied{x}{y} \neq 0$. Then 
\[ 0\neq \applied{x}{y} = \applied{Px}{y} = \applied{x}{P'y} \]
shows that $\fix(\mathscr{S}')$ separates $\fix(\mathscr{S})$.
Interchanging the roles of $X$ and $Y$ finishes the proof.
\end{proof}

Example~\ref{ex.nodecomp} shows that in Part (a)  of Proposition~\ref{p.weakerg} the assumption
that the fixed spaces are finite dimensional cannot be omitted.

\section{Average Schemes on $(C_b(E),\mathscr{M}(E))$}\label{s.eproperty}

Throughout this section, we fix a Polish space $E$ and work on the norming dual 
pair $(C_b(E), \mathscr{M}(E))$, which seems to be the most interesting norming dual pair for applications.
We will impose additional conditions on the average scheme $(\mathscr{S}, \mathscr{A})$ such that 
assertions (i) -- (iv) of Theorem \ref{t.erg} are equivalent. In view of  Proposition \ref{p.weakerg}(b), (iv) implies (iii), so the question is whether 
(iii) implies (i).

Examples \ref{ex:noconvergence} and \ref{ex.nodecomp}
show that this is not true without additional assumptions.

In this section we break the symmetry between $X=C_b(E)$ and $Y=\mathscr{M}(E)$
by imposing additional assumptions  on the semigroup on the function space. Under these assumptions we can show that
the assertions of Theorem \ref{t.erg} concerning the (adjoint)
semigroup on the space of measures are all equivalent.
\medskip

We start by recalling the definition of the strict topology on $C_b(E)$.
Denote by $\mathscr{F}_0(E)$ the space of all bounded functions $f$ on $E$ that vanish at infinity,
i.e.\ for every $\eps>0$ there is a compact set 
$K\subset E$ such that $\vert f(x)\vert < \eps$ for all $x\in K$.
The \emph{strict topology} $\beta_0$ on $C_b(E)$ is the locally convex topology generated by the set of
seminorms $\{ q_\varphi : \varphi \in \mathscr{F}_0(E)\}$ where $q_\varphi(f) \coloneqq \Vert \varphi f\Vert_\infty$.

The strict topology is consistent with the duality, i.e.\ $(C_b(E),\beta_0)'=\mathscr{M}(E)$,
see \cite[Thm 7.6.3]{jarchow1981}, and
it coincides with the compact open topology on norm bounded subsets of $C_b(E)$ \cite[Thm 2.10.4]{jarchow1981}.
Moreover, it is the Mackey topology of the dual pair $(C_b(E),\mathscr{M}(E))$ \cite[Thm 4.5, 5.8]{sentilles1972}, i.e.\
it is the finest locally convex topology on $C_b(E)$ which yields $\mathscr{M}(E)$ as a dual space.
In particular, $\cL(C_b(E),\sigma) = \cL(C_b(E),\beta_0)$, see \cite[21.4(6)]{koethe1969}.
\medskip

Now, we formulate and discuss the main condition we impose on an average scheme throughout this section.
Let $d$ be a complete metric $d$ that generates the topology of $E$
and denote by $\mathrm{Lip}_b(E,d)$ the space of all bounded Lipschitz continuous functions on $E$
with respect to $d$.
We assume the average scheme $(\mathscr{S},\mathscr{A})$ to satisfy the following.
\begin{hyp}
	\label{c.betastar}
	For every $f\in \mathrm{Lip}_b(E,d)$ the net $(A_\alpha f)_{\alpha\in\Lambda}$ clusters in $(C_b(E),\beta_0)$.
\end{hyp}
A priori, this requirement depends on the choice of the metric $d$. 
However, Hypothesis \ref{c.betastar} is necessary for the assertion of
Theorem \ref{t.eerg} which, in turn, is independent of the metric $d$.
Hence, under the other assumptions of Theorem \ref{t.eerg}, 
Hypothesis \ref{c.betastar} holds for some metric $d$ if and
only if it holds for every complete metric on $E$ that generates its topology.
Let us fix such a metric $d$ for the rest of this section.
\medskip

Let us compare Hypothesis \ref{c.betastar} with Assertion (\ref{t.erg.2}) of Theorem \ref{t.erg}.
Instead of assuming that $(A_\alpha f)$ clusters in $(C_b(E),\sigma)$ for every $f\in C_b(E)$,
we require that the net $(A_\alpha f)$ clusters with respect to the finer topology $\beta_0$,
but only for those functions $f$ with some additional regularity, namely for Lipschitz functions.

At first sight, Hypothesis \ref{c.betastar} seems to be rather technical. However, as already mentioned, it
is necessary for Theorem \ref{t.eerg} and, important from the point of view of applications,
it is implied by both the strong Feller property and the e-property, which are well-known assumptions
in the study of ergodic properties of Markov chains and semigroups cf. \cite{kps2010, sw2012}.
Let us discuss these relationships, starting with the e-property,
before continuing with our general theory.\medskip 

A family $\mathscr{T} \subset \cL(C_b(E))$ is said to have the
\emph{e-property} if the orbits $\{Tf:T \in \mathscr{T}\}$ are equicontinuous for all $f \in \mathrm{Lip}_b(E,d)$,
i.e.\ for all $x \in E$ and $\eps >0$ there exists a $\delta>0$
such that $|Tf(x) - Tf(y)| \leq \eps$ for all $T\in\mathscr{T}$ whenever $d(x,y) < \delta$.

A net $(T_i)_{i\in I} \subset \cL(C_b(E))$ is said the have the \emph{eventual e-property} if
there exists a $j \in I$ such that $\{ T_i : i\geq j\}$ has the e-property.

An average scheme $(\mathscr{S},\mathscr{A})$ is said to have the \emph{(eventual) e-property} 
if $(A_\alpha)_{\alpha \in \Lambda}$ has the (eventual) e-property.
\medskip

As an instructive example, consider the shift semigroup $\mathscr{S} = (S(t))_{t\geq 0}$ on $(C_b(\CR), \mathscr{M}(\CR))$,
given by $S(t)f(x) = f(t+x)$. Then $\mathscr{S}$ has the e-property (we will see below that this implies that 
also every average scheme $(\mathscr{S},\mathscr{A})$ has the e-property), since for 
$f \in \mathrm{Lip}_b(\CR, |\cdot|)$ we have
\[
 |S(t)f(x) - S(t)f(y)| = |f(t+x) - f(t+y)| \leq L |x-y|
\]
for all $x,y \in \CR$, where $L$ is the Lipschitz constant of $f$. Note, however, that 
this does \emph{not} imply that the orbit $S(t)f$ is equicontinuous for all $f \in C_b(\CR)$.

Before giving further examples, let us prove the mentioned result concerning the e-property of $\mathscr{S}$ and 
that of $\mathscr{A}$.

\begin{lem}
\label{l.semieprop}
Let $\mathscr{S} \subset \mathscr{L}(C_b(E),\sigma)$ be a semigroup which has the e-property.
Then every average scheme $(\mathscr{S},\mathscr{A})$ has the e-property.
\end{lem}
\begin{proof}
Let $f\in\mathrm{Lip}_b(E,d)$, $x\in E$ and $\eps>0$. By assumption, there exists $\delta>0$ such that
$\vert (Sf)(x) - (Sf)(y) \vert <\eps$ for all $S\in\mathscr{S}$ whenever $d(x,y)<\delta$.
Fix such a $y$.

By (AS2), given $\alpha \in \Lambda$ 
we find a function $g_y := \sum_{k=1}^n a_k S_kf \in \mathrm{co}(\mathscr{S}f)$
such that $| \dual{A_\alpha f -g_y}{\delta_x - \delta_y}| \leq \eps$. It follows that 
\begin{align*}
 |A_\alpha f (x) - A_\alpha f(y)| & \leq  |\dual{A_\alpha f -g_y}{\delta_x -\delta_y}| + |\dual{g_y}{\delta_x -\delta_y}|\\
& \leq  \eps + \sum_{k=1}^n a_k |S_kf(x) - S_kf(y)| \leq 2 \eps\, .
\end{align*}
Since $\eps$ and $\alpha$ where arbitrary, $\{A_\alpha f : \alpha \in \Lambda\}$ is equicontinuous.
Thus, $(\mathscr{S},\mathscr{A})$ has the e-property.
\end{proof}

Example \ref{ex:noconvergence} shows that if $\mathscr{S}$ does not have the e-property, there might be $\mathscr{A}$ 
and $\tilde{\mathscr{A}}$ such that $(\mathscr{S},\mathscr{A})$ has the
e-property whereas $(\mathscr{S}, \tilde{\mathscr{A}})$ does not have the e-property.\medskip 

\begin{rem}
\label{r.equicontcompact}
In view of the Arzel\`a-Ascoli theorem \cite[Thm 3.6]{khan1979} and (AS1), 
an average scheme $(\mathscr{S},\mathscr{A})$ has the (eventual) e-property
if and only if for every $f\in \mathrm{Lip}_b(E,d)$ the set 
$\{A_\alpha f: \alpha \in \Lambda\}$ (resp.\ $\{A_\alpha f: \alpha \geq \alpha_0\}$)
is relatively $\beta_0$-compact, equivalently, 
the set is relatively compact in the compact-open topology.

Thus, an average scheme with eventual e-property satisfies Hypothesis \ref{c.betastar}.
\end{rem}

We next discuss the strong Feller property which also implies Hypothesis \ref{c.betastar}.
Let us recall that an operator $T \in \cL (C_b(E), \sigma)$ is a kernel operator and thus has a unique extension to
an operator on $B_b(E)$. 
The operator $T$ is said to be \emph{strong Feller} if this extension maps $B_b(E)$ to $C_b(E)$. It is called \emph{ultra Feller}, if 
the extension maps bounded subsets of $B_b(E)$ to equicontinuous subsets of $C_b(E)$.

\begin{prop}
Let $\mathscr{S}\subset \cL(C_b(E),\sigma)$ be a semigroup such that some
 operator $S\in\mathscr{S}$ is strong Feller. Then every 
average scheme $(\mathscr{S},\mathscr{A})$ satisfies Hypothesis~\ref{c.betastar}.
\end{prop}

\begin{proof}
Let $f\in C_b(E)$ and $(\mathscr{S},\mathscr{A})$ be an average scheme.
Since $S$ is strong Feller, the operator $S^2\in\mathscr{S}$ is ultra Feller \cite[\S 1.5]{revuz1975}.
Hence, by (AS1) and the Arzel\`a-Ascoli theorem \cite[Thm 3.6]{khan1979}, the set
$\{ S^2 A_\alpha f : \alpha \in \Lambda \}$ is relatively $\beta_0$-compact. Thus,
every subnet of $(A_\alpha f)$ has a subnet $(A_{\alpha(\beta)} f)_{\beta\in J}$ 
such that $(S^2 A_{\alpha(\beta)})_{\beta\in J}$ converges in $(C_b(E),\beta_0)$.
By (AS3), 
\[ S^2 A_{\alpha(\beta)}f - A_{\alpha(\beta)}f \to 0 \]
in norm and thus with respect to the strict topology which is coarser.
\end{proof}

\medskip

Now we return to our main line of study.
Besides Hypothesis \ref{c.betastar}, we will impose further assumptions. First of all, we assume
that $\mathscr{S}$ is a \emph{Markovian semigroup}, i.e.\  every operator
$S \in\mathscr{S}$ is Markovian by which we mean that $S$ is positive and $S\one =\one$. 
If $k_S$ denotes the kernel associated with $S$, this is equivalent with the requirement that $k_S(x,\cdot)$ is a 
probability measure for all $x \in E$ and $S \in\mathscr{S}$.
Since the applications we have in mind for our theory concern transition semigroups of Markov chains or Markov 
processes, this assumption is rather natural.
We will call an average scheme $(\mathscr{S}, \mathscr{A})$
\emph{Markovian} if $\mathscr{S}$ is Markovian.
It follows from (AS2) that this implies that every operator $A_\alpha$ is Markovian, too.

Second, we assume that the directed index set $\Lambda$ of the averages $(A_\alpha)_{\alpha\in\Lambda}$ 
in this section has a cofinal subsequence.
This holds, for instance, in the classical situations
of Examples \ref{ex.cesaro}, \ref{ex.abel} and  \ref{ex.integrable}.
\medskip

We start with two auxiliary lemmas on $\beta_0$-equicontinuous operators.
	Let us recall that a family $\mathscr{T}$ of linear operators on a locally convex
	space $(X,\tau)$ is called \emph{$\tau$-equicontinuous}, if for every
	$\tau$-continuous seminorm $p$, there exists a $\tau$-continuous seminorm $q$ such that
	$p(Tx)\leq q(x)$ for every $T\in \mathscr{T}$ and $x\in X$.

In the case where $(X,\tau) = (C_b(E),\beta_0)$, we have the following characterization
of $\beta_0$-equicontinuous operators, see \cite[Prop 4.2]{kunze2009}. 
A family $\mathscr{T}\subset \cL(X,\beta_0)$ is $\beta_0$-equicontinuous
if and only if for every compact set $K\subset E$ and every $\eps>0$ there exists a compact set 
$L\subset E$ such that $\vert p_T \vert (x, E\backslash L) \leq \eps$
for all $x\in K$ and $T\in \mathscr{T}$ where $p_T$ denotes the kernel of $T$.

In what follows we will use that a family of Borel measures on $E$ is relatively $\sigma'$-compact 
if and only if it is tight and uniformly bounded in the variation norm,
cf.\ Theorems 8.6.7. and 8.6.8. of \cite{bogachev2007}.
Here, a family $\mathscr{F}\subset\mathscr{M}(E)$ of Borel measures is called \emph{tight} if for all $\eps>0$
there exists a compact set $K\subset E$ such that $\vert \mu\vert(E\backslash K)<\eps$ for all $\mu\in\mathscr{F}$.

\begin{lem}
\label{l.beta0equicont}
	Let $\{ T_j : j\in J \} \subset \cL(C_b(E),\sigma)$ be a family of Markovian operators
	that has the e-property.
	Suppose that the family $\{ p_j (x, \cdot ) : j \in J \}$ is
	tight for all $x\in E$ where $p_j$ denotes the kernel associated with $T_j$.
	Then the operators $\{ T_j : j \in J\}$ are $\beta_0$-equicontinuous.
\end{lem}
\begin{proof}
	Let $K\subset E$ be compact and $\eps>0$. By assumption, for every $x\in E$ there exists 
	a compact set $L_x \subset E$ such that $p_j (x, E\backslash L_x ) \leq \eps$ for all $j\in J$.
	We denote by 
	\[L_x^\eps \coloneqq  \{ x\in E : \dist(x, L_x) < \eps \} \]
	the open $\eps$-neighborhood of $L_x$ and define
	\[ f_x(y) \coloneqq \frac{\dist(x,E\backslash L_x^\eps)}{\dist(x,L_x)+\dist(x,E\backslash L_x^\eps)}.\]
	Then $f_x$ is Lipschitz continuous and $\mathds{1}_{L_x} \leq f_x \leq \mathds{1}_{L_x^\eps}$.
	Since the family $(T_j)$ has the e-property, there exist $\delta_x >0$ such that 
	\[ \vert (T_j f_x)(x) - (T_j f_x)(y) \vert < \eps\]
	for all $j\in J$ whenever $d(x,y) < \delta_x$. By the compactness of $K$, we find 
	$x_1,\dots, x_n \in E$ such that 
	\[ K \subset \bigcup_{k=1}^n \{ y\in E: d(x_k,y) < \delta_{x_k} \}.\]
	Let $L\coloneqq \cup_{k=1}^n L_{x_k}$ and $f(y)\coloneqq \max\{f_{x_1}(y),\dots,f_{x_n}(y)\}$.
	Since $L^\eps = \cup_{k=1}^n L_{x_k}^\eps$  we have $\mathds{1}_{L} \leq f \leq \mathds{1}_{L^\eps}$.
	Now fix an arbitrary $x\in K$ and choose $k\in \{1,\dots,n\}$ such that $d(x,x_k) < \delta_{x_k}$.
	Then we have
	\begin{align*}
		p_j(x,L^\eps) &\geq \int_E f(y) p_j(x,\mathrm{d} y) = (T_j f)(x)
		\geq (T_j f_{x_k})(x_k) -\eps \\
		&=  \int_E f_{x_k} (y) p_j(x_k,\mathrm{d} y) - \eps \geq p_j(x_k,L_{x_k})-\eps \geq 1-2\eps
	\end{align*}
	for every $j\in J$. It follows from \cite[Thm 3.2.2]{ethier1986} that
	the family 
	\[\{p_j(x, \cdot): j\in J,\, x\in K\}\]
	is tight.
	Since $K\subset E$ was arbitrary, the operators $T_j$ are $\beta_0$-equicontinuous by \cite[Prop 4.2]{kunze2009}.
\end{proof}

\begin{lem}
\label{l.orbitsequicont}
Let $\{T_j : j\in J \}\subset \cL(C_b(E),\beta_0)$ be a $\beta_0$-equicontinuous
family of operators with e-property.
Then $\{ T_j f : j\in J\}$ is equicontinuous for all $f\in C_b(E)$.
\end{lem}
\begin{proof}
	Since $\mathrm{Lip}_b(E,d)$ is a subalgebra of $C_b(E)$ which separates the points of $E$,
	it follows from the Stone-Weierstrass theorem
	that $\mathrm{Lip}_b(E,d)$ is dense in $(C_b(E),\beta_0)$, see \cite[Theorem 11]{fremlin1972}.
	Fix $f\in C_b(E)$ and $x_n,\, x\in E$ with $\lim x_n = x$. We show that $(T_j f)(x_n)$
	converges to $(T_j f)(x)$, uniformly in $j\in J$.
	This proves the equicontinuity of $\{ T_j f : j\in J\}$ in $x$.

	Consider the compact set $K\coloneqq\{x\}\cup \{x_n : n\in\N \}$ and the associated seminorm
	$p(h) \coloneqq \sup \{ \vert h(x) \vert : x\in K\} = \|h\varphi\|_\infty$ for $\varphi = \one_K$.
 Since the family $(T_j)$ is $\beta_0$-equicontinuous, 
 there exists a $\beta_0$-continuous seminorm $q: C_b(E) \to [0,\infty)$
	such that $p(T_j h) \leq q(h)$ for all $h\in C_b(E)$ and $j\in J$.
	Now given $\eps>0$, pick $g\in \mathrm{Lip}_b(E,d)$ such that
	$q(f-g) \leq \eps$. Since the family $\{ T_j g : j\in J\}$ is equicontinuous, 
there exists $n_0 \in\N$ such that
	\[ \vert (T_j g)(x_n) - (T_j g)(x) \vert \leq \eps\]
	for all $n\geq n_0$ and all $j\in J$. This implies that
	\[ 
		\vert (T_j f)(x_n) - (T_j f)(x) \vert 
		\leq 2q(f-g) + \vert (T_j g)(x_n) - (T_j g)(x) \vert \leq 2\eps + \eps
	\]
	for all $n\geq n_0$ and $j\in J$.
\end{proof}

Now we prove the main result of this section.

\begin{thm}\label{t.eerg}
Let $(\mathscr{S},\mathscr{A})$ be a Markovian average scheme on $(C_b(E),\mathscr{M}(E))$
that satisfies Hypothesis \ref{c.betastar} and
suppose that there exists an increasing cofinal sequence in $\Lambda$. Then the following assertions
are equivalent.
\begin{enumerate}[(i)]
	\item $(\mathscr{S},\mathscr{A})$ is weakly ergodic and $\beta_0\minus\lim_\alpha A_\alpha f = P f$
	for all $f\in C_b(E)$ where $P$ is the ergodic projection.
	\item For every $x\in E$ the net $(A'_\alpha \delta_x)$ has a $\sigma'$-cluster point.
	\item $\fix(\mathscr{S}')$ separates $\fix(\mathscr{S})$.
	\item $\mathscr{M}(E) = \fix(\mathscr{S}') \oplus \overline{\lh}^{\sigma'} (I-\mathscr{S}')$.
\end{enumerate}
\end{thm}
\begin{proof}
	By Theorem \ref{t.erg}, (i) implies (ii) -- (iv).
	Moreover, the implication (iv) $\Rightarrow$ (iii) is part of the proof of Proposition \ref{p.weakerg}
	and (iii) follows from (ii) as in the proof of Theorem~\ref{t.erg}.
	Hence, it remains to prove that (iii) implies (i) to complete the proof.
	\smallskip

	Let us assume that
	$\fix(\mathscr{S}')$ separates $\fix(\mathscr{S})$. We denote by
$(\alpha_n) \subset \Lambda$ an increasing cofinal sequence.  As
$(A_\alpha f)$ clusters in $(C_b(E),\beta_0)$, Lemma \ref{l.convergence} yields that
\[ \beta_0\minus\lim_\alpha A_\alpha f = \beta_0\minus\lim_{n\to\infty} A_{\alpha_n} f\in C_b(E) \]
exists for all $f\in \mathrm{Lip}_b(E,d)$.
Fix a non-negative measure $\mu \in \mathscr{M}(E)$. Then the scalar sequence
$\applied{f}{A'_{\alpha_n}\mu}=\applied{A_{\alpha_n}f}{\mu}$
converges as $n\to\infty$ for all $f\in \mathrm{Lip}_b(E,d)$.
Thus, by \cite[Cor 8.6.3]{bogachev2007}, the family $\{ A_{\alpha_n}' \mu : n\in\N\}$
is tight and Prohorov's theorem \cite[Thm 8.6.2]{bogachev2007} implies that, passing to a subsequence,
$(A'_{\alpha_n}\mu)_{n\in\N}$ converges weakly to some measure $\tilde\mu \in \mathcal{M}(E)$.
Altogether, this shows that
\[ \lim_\alpha \applied{f}{A'_\alpha \mu} = \lim_{\alpha} \applied{A_\alpha f}{\mu} 
= \lim_{n\to\infty} \applied{A_{\alpha_n}f}{\mu} = \lim_{n\to\infty} \applied{f}{A'_{\alpha_n}\mu}
= \applied{f}{\tilde\mu}\]
for all $f\in \mathrm{Lip}_b(E,d)$.
Since $E$ is separable, the set $\mathrm{Lip}_b(E,d)$ is convergence determining by \cite[Prop 3.4.4]{ethier1986}, 
hence it follows that $\sigma'\minus\lim_\alpha A'_\alpha \mu = \tilde\mu$.
Decomposing a general measure in positive and negative part yields that 
$\sigma'\minus \lim_\alpha A'_\alpha \mu \in \mathscr{M}(E)$ exists for every $\mu\in\mathscr{M}(E)$.

By Hypothesis \ref{c.betastar} and Lemma \ref{l.seqcompact},
the set $\{A_{\alpha_n} f : n\in\N\}$ is relatively $\beta_0$-compact for 
every $f\in \mathrm{Lip}_b(E,d)$, which implies by the Arzel\`a-Ascoli theorem
(cf.\ Remark \ref{r.equicontcompact})
that the family $\{A_{\alpha_n} : n\in\N\}$ has the e-property.
Now, we infer from Lemma \ref{l.beta0equicont} 
that the operators $\{ A_{\alpha_n} : n\in\N\}$  are $\beta_0$-equicontinuous.
By Lemma \ref{l.orbitsequicont} this implies that the orbits $\{A_{\alpha_n} f : n\in\N\}$ 
 are equicontinuous for all $f\in C_b(E)$.
Using the Arzel\`a-Ascoli theorem again, it follows that
$\{A_{\alpha_n} f : n\in\N\}$ is relatively $\beta_0$-compact for all $f\in C_b(E)$.


Now we conclude from Theorem \ref{t.erg} that the average scheme $(\mathscr{S},(A_{\alpha_n}))$
is weakly ergodic with an ergodic projection $P \in \cL(C_b(E),\sigma)$.
Since $(\alpha_n)$ was arbitrary and $P$ does not depend on the averages $(A_{\alpha_n})$, 
even $(\mathscr{S},\mathscr{A})$
is weakly ergodic. 
Finally, the $\beta_0$-convergence of $(A_\alpha f)$  follows from Lemma \ref{l.convergence}.
\end{proof}

\begin{rem}
Assume that $\Lambda = \CN$ or $\Lambda = [0, \infty)$ in their natural ordering and that $\alpha \mapsto A_\alpha f$ is 
continuous as a map with values in $(C_b(E), \beta_0)$. If $\Lambda = \CN$, this is always the case, for $\Lambda = [0, \infty)$ 
this is true for the Ces\`aro averages $A_t$ of a semigroup  $\mathscr{S} =(S(t))_{t \geq 0}$ on $(C_b(E), \mathscr{M}(E))$ 
which has $\beta_0$-continuous orbits.

If $(\mathscr{S}, \mathscr{A})$ is weakly mean ergodic, and $A_\alpha f$ converges to $Pf$ with respect to $\beta_0$, 
then the means $\{A_\alpha : \alpha \in \Lambda \}$ are $\beta_0$-equicontinuous. Indeed, in this case the function 
$F: \Lambda \cup \{\infty\} \to \cL (C_b(E), \sigma)$, defined by $F(\alpha ) = A_\alpha$ for $\alpha \neq \infty$ and 
$F(\infty) = P$ is strongly $\beta_0$-continuous, whence the equicontinuity follows from \cite[Lemma 3.8]{kunze2009}.

We are thus in the situation of mean ergodic theorems on locally convex spaces \cite{eberlein1949, sato1978}. Note, however, that 
in Theorem \ref{t.eerg} we do not a priory assume $\beta_0$-equicontinuity since, as the example of the shift semigroup 
shows, Hypothesis \ref{c.betastar} alone does not imply equicontinuity of the averages.
\end{rem}

\begin{rem}
It is immediate that in the situation of Theorem \ref{t.eerg} if $(\mathscr{S}, \mathscr{A})$ is weakly mean ergodic, then 
 any average scheme $(\mathscr{S}, \mathscr{B})$ which satisfies Hypothesis \ref{c.betastar} is also weakly mean ergodic.
\end{rem}

\section{Counterexamples}\label{s.examples}

We conclude this article with a collection of examples that illustrate that the results obtained in 
Section \ref{s.convergence} are optimal.
Our first example shows that even if $A_\alpha x$ $\sigma$-converges for all $x\in X$, it can happen that
on $Y$ the averages $A'_\alpha y$ do not $\sigma'$-converge for some $y\in Y$.
\begin{example}
\label{ex:Pnotcontinuous}
Consider the norming dual pair $(\ell^1, c_0)$ and the power-bounded operator $S: \ell^1 \to \ell^1$, defined by 
$S: (x_1, x_2, x_3, \ldots) \mapsto (x_1 + x_2, x_3, \ldots)$. Then the adjoint operator is given by
$S^* (y_1, y_2, \ldots) = (y_1, y_1, y_2, \ldots)$. 
In particular, $S^*c_0 \subset c_0$ so that $S \in \cL (\ell^1, \sigma)$. Since 
\[
 S^n(x_1,x_2,\ldots) = (\sum_{j=1}^n x_j, x_{n+1}, x_{n+2}, \ldots) 
\]
clearly $\sigma$-converges to $(\sum_{j=1}^\infty x_j, 0,0, \ldots)$, the $\sigma$-limit of the Ces\`aro averages
\[ A_n \bx = \frac{1}{n} \sum_{k=0}^{n-1} S^k \bx \]
exists for all $\bx \in \ell^1$.
However, in this situation the 
ergodic projection $P : \bx \mapsto (\sum_{j=1}^\infty x_j, 0, 0, \ldots)$ does not respect the duality, i.e.\ 
$P \not\in \cL (\ell^1, \sigma)$. Indeed, for $\bx  \in \ell^1$ and $\by \in \ell^\infty = (\ell^1)^*$ we have
\[
 \applied{P\bx}{\by} =y_1\sum_{j=1}^\infty x_1 = \applied{\bx}{y_1\one}
\]
whence $P^*c_0 \not\subset c_0$. By Theorem \ref{t.erg}, the
$\sigma'$-limit of the adjoint averages $A'_n \by$ does not exist for some $\by \in c_0$.
\end{example}

Our next example shows that if $(\mathscr{S},\mathscr{A})$ is an average scheme so that both $X$ and $Y$ can be
decomposed as in (\ref{t.erg.4}) of Theorem \ref{t.erg}, the average scheme is not necessarily weakly ergodic.
In fact, we present two different averages $\mathscr{A}$ and $\tilde{\mathscr{A}}$ for the same semigroup $\mathscr{S}$
such that $(\mathscr{S},\mathscr{A})$ is weakly ergodic whereas for $(\mathscr{S},\tilde{\mathscr{A}})$ only the weaker
convergence properties of Proposition \ref{p.weakerg} (\ref{i.weakerg.b}) hold.

\begin{example}\label{ex:noconvergence}
We consider the set $E = \CZ\cup \{\infty\}$, where every point in $\CZ$ is isolated, whereas the neighborhoods of the 
extra point $\infty$ are exactly the sets which contain a set of the form $\{ n, n+1, \ldots\}\cup\{\infty\}$ for 
some $n\in\CZ$. Note that $E$ is homeomorphic with $\{-n\,:\, n\in\CN\} \cup \{1-\frac{1}{n}\,:\, n\in\CN\} \cup\{1\}$
endowed with the topology inherited from $\CR$, thus $E$ is Polish.

We work on the norming dual pair $(C_b(E),\mathscr{M}(E))$. Note that a function $f : E \to\CR$ is continuous 
if and only if $\lim_{n\to\infty}f(n) = f(\infty)$ and that $\mathscr{M}(E) = \ell^1(E)$. Consider the 
semigroup 
$\mathscr{S}:= \{S^n: n\in \CZ\}$, where
\[
 (Sf)(k) = f(k+1)\quad \mbox{for } k \in\CZ\qquad \mbox{and}\qquad (Sf)(\infty) = f(\infty)\, .
\]
Then $\fix (\mathscr{S}) = \{c\one_E: c \in \CR\}$ and $\fix (\mathscr{S}') = \{ c \delta_\infty: c \in\CR\}$.
In particular, the fixed spaces separate each other and are finite dimensional so that, as a consequence of 
Proposition~\ref{p.weakerg} (a), we have
\[
 C_b(E) = \fix (\mathscr{S}) \oplus \overline{\lh}^\sigma \rg (I -\mathscr{S})
\quad\mbox{and}\quad 
\mathscr{M}(E) = \fix (\mathscr{S}') \oplus \overline{\lh}^{\sigma'} \rg (I -\mathscr{S}').
\]
We now consider the average schemes $A_n$ and $\tilde{A}_n$, defined by
\[
 A_nx \coloneqq \frac{1}{n}\sum_{k=0}^{n-1}S^kx\quad\mbox{and}\quad 
\tilde{A}_nx \coloneqq \frac{1}{n}\sum_{k=0}^{n-1}S^{-k}x.
\]
That $A_n$ and $\tilde{A}_n$ are indeed average schemes is proved as in Example~\ref{ex.cesaro}.
Defining $Y_0 := \fix (\mathscr{S}') \oplus \overline{\lh}^{\|\cdot\|}\rg (I-\mathscr{S}')$, it
follows from Proposition~\ref{p.weakerg} (b) that 
\[\sigma (C_b(E), Y_0) \minus\lim_{n\to\infty}A_nf = \sigma (C_b(E),Y_0)\minus\lim_{n\to\infty}\tilde{A}_nf
 = Pf = f(\infty)\one_E
\]
for all $f\in C_b(E)$. 
Actually, using that $f(n) \to f(\infty)$ as $n\to \infty$ for all $f \in C_b(E)$, it is easy to see that
$A_nf \to f(\infty)\one_E$ pointwise for all $f \in C_b(E)$, hence with respect to $\sigma (C_b(E),\mathscr{M}(E))$.

However, $\tilde A_n f$ does not $\sigma(C_b(E),\mathscr{M}(E))$-converge to $f(\infty)\mathds{1}_E$ for some $f\in C_b(E)$.
Indeed, for $f := \one_{\CN\cup{\{}\infty{\}}} \in C_b(E)$ 
the sequence $\tilde{A}_nf$ converges pointwise to the function $\one_{{\{}\infty\}}$ which is not continuous.

This shows that in Proposition~\ref{p.weakerg} we cannot expect better convergence than with respect to 
$\sigma (X,Y_0)$. On the other hand, this example also shows that even if both $X$ and $Y$ have ergodic decompositions,
it can depend on the average scheme how strong the convergence to the ergodic projection is.

Let us also note that the average scheme $A_n$ has the e-property, whereas $\tilde{A}_n$ does not have the e-property.
For the e-property, the only point of interest is the point $\infty$, as all other points of $E$ are isolated. First note 
that in this case $C_b(E) = \mathrm{Lip}_b(E,d)$. Given $f \in C_b(E)$ and $\eps >0$, we find $n_0$ such that 
$|f(n) - f(\infty)| \leq \eps$ for all $n\geq n_0$. Consequently, we also have $|S^kf(n) - S^kf(\infty)| = |f(n+k) - f(\infty)|
\leq \eps$ for all $n\geq n_0$ and all $k \geq 0$. Thus $|A_kf(n) - A_kf(\infty)| \leq k^{-1}\sum_{j=0}^{k-1}
|f(n+j) -f(\infty)| \leq \eps$ for all $k \in \CN$ and all $n\geq n_0$, i.e.\ $\{A_kf:k\in\CN\}$ is equicontinuous.
On the other hand, $\tilde{A}_n$ cannot have the e-property, since in this case it would follow from Theorem~\ref{t.eerg}
that $\tilde{A}_nf \to Pf$ pointwise, which was seen to be wrong above.
\end{example}

We have seen in Proposition \ref{p.weakerg} (\ref{i.weakerg.a}) that if $\fix(\mathscr{S})$ and $\fix(\mathscr{S}')$ separate 
each other and are finite dimensional, then both $X$ and $Y$ can be decomposed as in (\ref{t.erg.4}) of Theorem \ref{t.erg}.
Our last example shows that this is not true for infinite dimensional fixed spaces.

\begin{example}
\label{ex.nodecomp}
In the following we construct a positive,
contractive and $\sigma$-continuous operator $S$ on the norming dual pair 
$(C_b(E),\mathscr{M}(E))$ such that, for $\mathscr{S}:=\{S^n:n\in\CN_0\}$, we have 
\[ \mathscr{M}(E)\not= \fix(\mathscr{S}') \oplus \overline{\rg}^\sigma(I-\mathscr{S}')\]
while the fixed spaces of $\mathscr{S}$ and $\mathscr{S}'$ separate each other. 

For $n\in\N$ let $K_n\coloneqq \{0,\dots,n\} \times \{1/n\}$ and $K_0 \coloneqq \N_0\times \{0\}$.
On the set $E \coloneqq \bigcup_{n\in\N_0} K_n$
endowed with the topology inherited from $\R^2$,
we consider the continuous mapping $\varphi : E\to E$, given by
\[ \varphi((k,1/n)) \coloneqq \begin{cases}
((k+1), 1/n ) & n\in\N,\, k\in\{0,\dots,n-1\}\\
(0,1/n) & n\in\N,\, k=n \end{cases} \]
and $\varphi(k,0) \coloneqq (k+1,0)$ for all $k\in\N$. Thus on each $K_n$ the map $\varphi$ shifts to the right and for 
$n \neq 0$ the point $(n,\frac{1}{n})$ is mapped to $(0,\frac{1}{n})$, see Figure~\ref{fig.omega}.
\begin{figure}[h!]
	\setlength{\fboxsep}{10pt}
	\fbox{\includegraphics[width=0.5\textwidth]{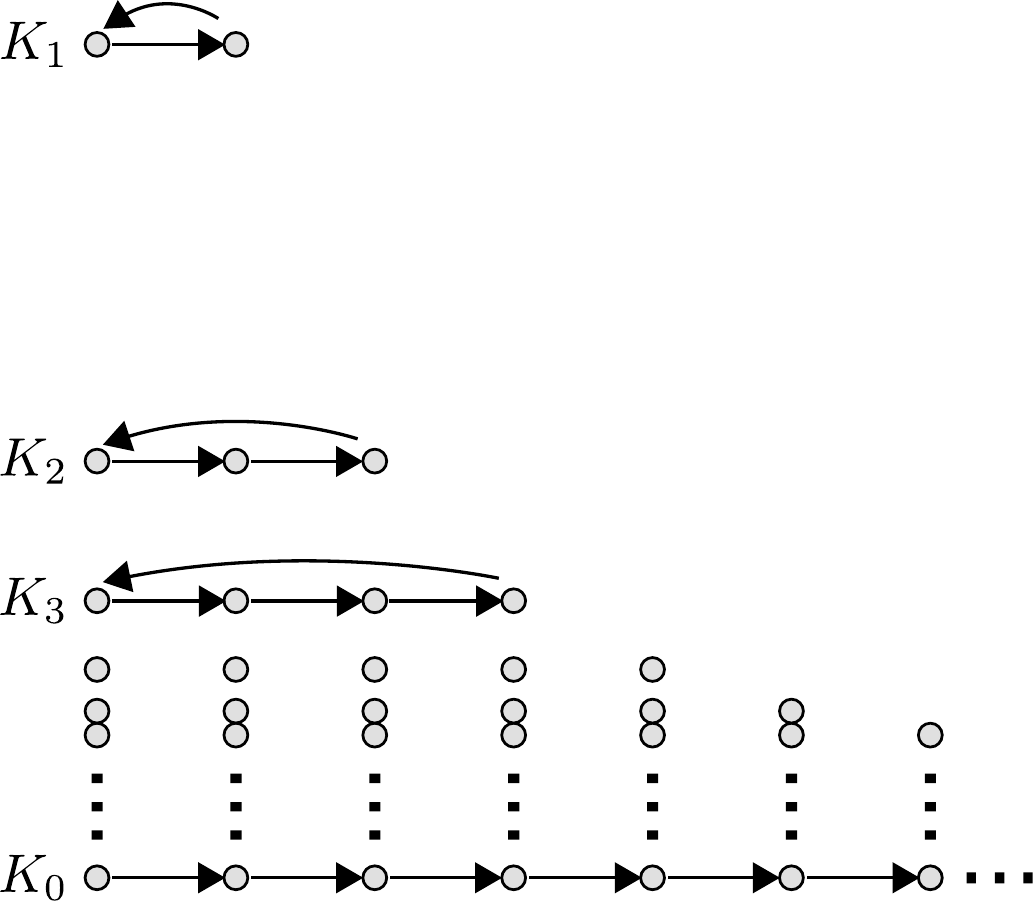}}
	\caption{Transformation of $E$ by the action of $\varphi$}
	\label{fig.omega}
\end{figure}

Let $S$ denote the induced operator on $C_b(E)$, defined as $Sf \coloneqq f\circ \varphi$.
It is easy to see that 
\[ \fix(\mathscr{S}) = \fix(S) = \Big\{ \sum_{n=0}^\infty a_n\one_{K_n}\,:\,
 \lim_{n\to\infty}a_n = a_0\Big\}
\]
On the other hand, 
\[ \fix(\mathscr{S}') = \fix(S') = \Big\{ \sum_{n=1}^\infty a_n\zeta_n\,:\,
 (a_n) \in \ell^1\Big\},
\]
where $\zeta_n$ denotes counting measure on $K_n$ with the normalization $\zeta_n(K_n)=1$.
It thus follows that the fixed spaces separate each other. 

We now show that 
\[ \delta_0 \not\in \fix (\mathscr{S}') \oplus \overline{\lh}^{\sigma'}\rg(I-\mathscr{S}').\]
Aiming for a contradiction, let us assume that there exists a sequence $(a_n)\in \ell^1$ and
a net $(\mu_\alpha)_\alpha \subset \lh\, \rg (I-\mathscr{S}')$, $\sigma'$-converging to $\mu \in \mathscr{M}(E)$, such that
\[ \delta_0 = \sum_{n=1}^\infty a_n \zeta_n + \mu.\]
Since $\one_{K_n}$ is a fixed 
point of $\mathscr{S}$ and $\mu_\alpha$ belongs to $\lh \,\rg (I-\mathscr{S}')$, we have
 $\dual{\one_{K_n}}{\mu_\alpha} = 0$
for all $n\in\CN$ and all $\alpha$. Thus also $\dual{\one_{K_n}}{\mu} = 0$. It follows that
\[ 0 = \applied{\mathds{1}_{K_n}}{\delta_0} = a_n \zeta_n(K_n) + \mu(K_n) = a_n\]
for all $n\in\N$ and hence $\delta_0 = \mu$. Since $\one_E \in \fix{\mathscr{S}}$, the contradiction
\[ 1 = \dual{\one_E}{\delta_0} = \lim_{\alpha} \dual{\one_E}{\mu_\alpha} = 0\]
follows.
\end{example}

\bibliographystyle{abbrv}
\bibliography{analysis}

\end{document}